\documentclass[12pt,leqno]{amsart}

\usepackage[margin=2.5cm]{geometry} 

\usepackage{graphicx} 
\usepackage[utf8]{inputenc}
\usepackage[english]{babel}
\usepackage[T1]{fontenc}
\usepackage{amssymb}
\usepackage{mathtools}
\usepackage{amsfonts}
\usepackage{tikz}
\usepackage{tikz-cd}
\usepackage{amsmath}
\usepackage{amsthm}
\usepackage{mathrsfs}
\usepackage{changepage}
\graphicspath{{images/}}
\usepackage{fancyhdr}
\usepackage{csquotes}

\usepackage[backend=biber, sortcites=true, maxbibnames=50]{biblatex}
\usepackage{hyperref}
\hypersetup{
	colorlinks=true,
	linkcolor=blue,
	citecolor=blue,
	urlcolor=blue,
	pdftitle={}
}

\newtheorem{theorem}{Theorem}  
\newtheorem{cor}[theorem]{Corollary}  

\theoremstyle{plain}
\newtheorem{thm}{Theorem}
\newtheorem{lemma}[thm]{Lemma}
\newtheorem{corollary}[thm]{Corollary}

\newtheorem{prop}[thm]{Proposition}

\theoremstyle{definition}
\newtheorem{definition}[thm]{Definition}

\newtheorem{remark}[thm]{Remark}
\newtheorem{exmp}[thm]{Example}

\numberwithin{equation}{section}
\numberwithin{thm}{section}

\usepackage{setspace}

\theoremstyle{remark} 
\newtheorem*{ack}{Acknowledgements}

\title[Regularity of Resolutions and Limits of Manifolds]{Regularity of Resolutions and Limits of Manifolds with a Uniform Contractibility Function}
\author{Mohammad Alattar and Lewis Tadman}
\date{\today}

\address[Alattar]{Department of Mathematical Sciences, Durham University, United Kingdom}
\email{{mohammad.al-attar@durham.ac.uk}}
\address[Tadman]{Department of Mathematical Sciences, Durham University, United Kingdom}
\email{lewis.tadman@durham.ac.uk}

\begin{document}

\begin{abstract}
 In this paper, we give a short and self-contained proof to a 1991 conjecture by Moore concerning the structure of certain finite-dimensional Gromov--Hausdorff limits, in the ANR setting. As a consequence, one easily characterizes finite dimensional limits of PL-able or Riemannian $n$-manifolds with a uniform contractibility function. For example, one can define for any compact connected metric space that is a resolvable ANR homology manifold of covering dimension at least 5, an obstruction, which vanishes if and only if the homology manifold can be approximated in the Gromov--Hausdorff sense by PL-manifolds of the same dimension and with a uniform contractibility function. Further, it provides short proofs to certain well known results by reducing them to problems in Bing topology.  
 We also give another proof using more classical arguments that yield more structural information. We give several applications to the theory of homology manifolds, Alexandrov spaces, Wasserstein spaces and a generalized form of the diffeomorphism stability conjecture.
\end{abstract}

\maketitle

\section{Introduction}
Determining when certain spaces can be approximated by Riemannian manifolds under various constraints is a central question in Riemannian and metric geometry (see for example \cite{kawamura, hegenbarth-repovs,hegenbarthrepovsII,  cassorla,kapovitch-regularity,kapovitchregularityII,aleksandrovandzalgaller,Alexandrov,nikolaev,NikolaevI,DottI,DottII,DottIII,PhilippReiser}). The purpose of this paper is to identify when certain topological spaces can be approximated in the Gromov--Hausdorff topology by PL or Riemannian manifolds with a uniform contractibility function (see Definition \ref{contractibility function}). Classes of spaces admitting a uniform contractibility function have been well studied in comparison geometry. For example, Grove, Petersen and Wu showed that the class of closed (i.e., compact and without boundary) Riemannian $n$-manifolds with sectional curvature uniformly bounded below, diameter uniformly bounded above, and volume uniformly bounded below admits a uniform contractibility function \cite{GPW}. If one replaces ``sectional'' curvature with ``Ricci'' curvature, then Zhu \cite{zhu} showed that in dimension $3$, a result in the direction of Grove--Petersen--Wu holds true (cf. \cite{Anderson} and see \cite{bruè2024topological} for a recent generalization to Zhu's results and more). For spaces with upper curvature bounds, we refer the reader to the paper \cite{LytchakNagano2}.

To understand the structure of Gromov--Hausdorff limits of finite-dimensional compact metric spaces admitting a uniform contractibility function, it is necessary to understand the theory of cell-like maps (see Definition \ref{Cell-like maps}). Indeed, suppose $X$ is a Gromov--Hausdorff limit of closed  metric topological $n$-manifolds with a uniform contractibility function.  Grove, Petersen, and Wu \cite{GPW,Grove-Petersen-Wu-erratum} showed that if $X$ is finite dimensional, then it is a cell-like image of some manifold. If one only demands that $X$ is a limit of $n$-dimensional compact absolute neighborhood retracts (or ANR for short, see Definition \ref{anr}) with a uniform contractibility function, Moore \cite{Moore}  showed that $X$ is a cell-like image of some $n$-dimensional compactum. Ferry \cite{Ferrypolyhedra} showed if $X$ is a Gromov--Hausdorff limit of compact metric spaces with a uniform contractibility function, and with covering dimension bounded above, then $X$ is a cell-like image of some polyhedron.

In general, under non-collapsed Gromov--Hausdorff convergence of compact metric $n$-manifolds with a uniform contractibility function (see Definition \ref{contractibility function}), the limit space was shown to be, in this general setting, a resolvable homology manifold (see for example \cite{GPW,Grove-Petersen-Wu-erratum,Wu} and cf.\  \cite[p.\ 88]{problemsinhomologymanifolds}). That is, the limit space was shown to be the cell-like image of \textit{some} manifold. The next theorem shows that the manifold witnessing the resolution can be made more precise. In turn, we obtain immediately as a corollary, a proof to a  conjecture of Moore \cite{Moore} in the finite dimensional case.

\begin{theorem}
\label{main thm}
Fix $n\geq 4$. Let $\{X_i^n\}_{i=1}^{\infty}$ be a sequence of closed metric topological $n$-manifolds with a uniform contractibility function converging to $X$ in the Gromov--Hausdorff sense. If $X$ is an $\mathrm{ANR}$, then $X$ is a cell-like image of  $X_i$ for all $i$ sufficiently large. Moreover, the cell-like map can be taken to be an open map.
\end{theorem}

We present two proofs of Theorem \ref{main thm} using different, but intersecting, techniques.  The first proof is short, and uses an ``alternating trick'' which has found applications elsewhere (see for example Section \ref{section examples}).

The second proof modifies the Grove--Petersen--Wu stability theorem \cite{GPW,Grove-Petersen-Wu-erratum} by incorporating a convergence argument to construct resolutions from controlled surgery \cite{topologyofhomologymanifolds,topofhommfldserratum}.  This proof refines the first proof by giving a more concrete construction of the cell-like map appearing in the Theorem above. In particular, it gives more insight on the relationship between the limiting space and the approximating manifolds.

The ideas surrounding the convergence technique also appear in a slightly different form in a proof of the Reifenberg theorem for metric spaces due to Cheeger and Colding \cite{CheegerColdingI,CheegerColdingII,CheegerColdingIII} (see \cite{Reifenbergexposition} for an exposition of this theorem).

Theorem~\ref{main thm} yields a proof to a conjecture by Moore in the finite dimensional case \cite[p.\ 418]{Moore}. In what follows, dimension will be understood to be covering dimension. Further, we will denote by $\mathrm{LGC}(n,\rho)$ the class of compact metric spaces with a uniform contractibility function that are $n$-locally geometrically contractible (see Definition \ref{LGC} for the precise definition). We denote by $\mathrm{LGC}(\rho)$ the union of all the $\mathrm{LGC}(n,\rho)$.
\begin{cor}
\label{conjecturebymooreII}
Fix $n\geq 4$. Assume $\omega\colon [0,1]\rightarrow \mathrm{LGC}(n,\rho)$ is a continuous path such that $\omega(t)$ for $0\leq t<1$ is a closed $n$-dimensional manifold. If $\omega(1)$ is an $\mathrm{ANR}$ then $\omega(1)$ is a cell-like image of $\omega(0)$.
\end{cor}

The results above yield short proofs to well-known results that concern certain non-collapsed limits of metric manifolds. More precisely, they allow one to divorce the controlled homotopy theory surrounding ``stability'' arguments, and instead allow one to appeal more directly to classical decomposition space theory. For example, a common argument in stability arguments is to consider a sequence of closed metric $n$-manifolds with a uniform contractibility function $\{X_i^n\}_{i=1}^{\infty}$ that converge to some finite dimensional space $X$ (see for example \cite[Corollary 7.2]{LytchakNagano2} and cf.\ \cite{fujiokaandgu, GPW,Grove-Petersen-Wu-erratum}).  Provided the space $X$ is a manifold, a fundamental theorem due to Petersen \cite{petersenfinitenesstheoremformetricspaces}, in conjunction with the $\alpha$-approximation theorem due to Chapman and Ferry \cite{ChapmanFerry}, implies that $X_i$ and $X$ are homeomorphic for sufficiently large $i$. Assuming $n\geq 5$, one can obtain the same consequence more directly using Theorem \ref{main thm}. For example, if $X$ satisfies the disjoint disk property, then $X$ is homeomorphic to $X_i$ by Edward's theorem \cite{Robert-Edwards}. Moreover, the homeomorphisms can be chosen to approximate the cell-like map uniformly (see Theorem \ref{Edward's theorem}). We mention other similar consequences below.

For the next corollary, denote by $\mathcal{M}(n,\rho)$ the class of closed topological metric $n$-manifolds with a uniform contractibility function $\rho$. We say that a topological manifold \textit{admits a PL structure} if it is homeomorphic to a PL-manifold.

\begin{cor}
\label{approxbyPL}
 Let $(X^n,d)$ be an $n$-dimensional compact metric space. If $n\geq 4$, then the following statements are equivalent:

 \begin{enumerate}
     \item $X$  is a cell-like image of a closed PL $n$-manifold.
     \item There exist closed  manifolds admitting PL-structures  $\{X_i^n\}_{i=1}^{\infty}$ in $\mathcal{M}(n,\rho)$ such that $X_i\rightarrow_{GH}(X,d)$.
\end{enumerate}
\end{cor}

The implication $(1)\implies(2)$ follows immediately from Corollary \ref{refinementofmoore'sthm}. The implication $(2)\implies (1)$ follows immediately from Theorem \ref{main thm}.

Corollary \ref{approxbyPL} above can be rephrased by replacing $(1)$ with the vanishing of a certain obstruction in $H^4(X,\mathbb{Z}_2)$. Before proceeding, recall that a metric space $X$ with covering dimension $n$ is \emph{resolvable} if there exists an $n$-dimensional manifold $M$ and a cell-like map $f\colon M \to X$ (see Definition~\ref{def-of-resolution}). Examples of resolvable homology manifolds include, for example, Alexandrov spaces that are homology manifolds \cite{Wu}, spaces with upper curvature bounds that are homology manifolds \cite{LytchakNagano2} and homology manifolds that are locally Busemann spaces of non-positive curvature \cite{fujiokaandgu}. 

\begin{cor}
\label{obstructionvanishing}
Let $(X^n,d)$ be a compact connected metric space that is a finite dimensional ANR resolvable homology manifold with $n\geq 5$. Then there exists $\Delta(X^n)\in H^4(X^n;\mathbb{Z}_2)$ such that the following statements are equivalent:

\begin{enumerate}
    \item $\Delta(X^n)=0$.
    \item There exist metric topological manifolds $\{M_i^n\}_{i=1}^{\infty}$ admitting PL structures and a uniform contractibility function such that $M_i^n\rightarrow_{GH}(X^n,d)$.

\end{enumerate}
\end{cor}
To the best of the authors' knowledge, the Corollary above is the first result that relates topological obstructions to existence of certain Gromov--Hausdorff approximations. The obstruction above is easy to define. The non-trivial part is to show that it is well-defined. This will rely on the work of Quinn \cite{QuinnEndsofMaps}. In low dimensions, provided the metric $d$ is geodesic, then the PL-able manifolds can be taken to be Riemannian (see Remark \ref{low dimensions}).

Next, we investigate when a space $X$ can be approximated by closed Riemannian $n$-manifolds with a uniform contractibility function in higher dimensions. We restrict our attention to dimensions at least $5$. In dimensions at least $3$,  Ferry and Okun \cite{ferryandborus} obtained results in this direction and we will rely on their work. In dimension $2$, Dott has analyzed the situation at length \cite{DottI,DottII,DottIII}.

\begin{cor}
\label{corollary smoothable}
Let $(X^n,d)$ be a compact geodesic space with covering dimension $n$. If $n\geq 5$, then the following assertions are equivalent:

\begin{enumerate}
    \item $X$ is resolvable by a smooth manifold.
    \item $X\times \mathbb{R}^k$ admits a smooth structure for some $k\geq 2$.
    \item There exist closed connected Riemannian manifolds $\{X_i^n\}_{i=1}^{\infty}$ in $\mathcal{M}(n,\rho)$ such that $X_i\rightarrow_{GH}(X,d)$.
\end{enumerate}

\end{cor}

 Note that there exists spaces $X$ that are not manifolds but $X\times \mathbb{R}^k$ is a manifold. In fact, it is  even the case that some of these spaces  have no manifold points whatsoever (see Example \ref{GhastlyExamples}). In dimension $4$, the analogous statement to Corollary \ref{corollary smoothable} is false (see Example \ref{counterexampleindimension4}). We point out that one cannot strengthen the uniform contractibility condition in item $(3)$ of Corollary \ref{corollary smoothable}  to a condition where the approximating manifolds have a uniform lower curvature bound, even if $X$ is assumed to be a topological manifold with a lower sectional curvature bound in the sense of Alexandrov (see Example \ref{vitali}).

For the next corollary, we need to define the notion of being \emph{CE-related}, which was introduced in \cite{DFW}. We say that two compact metric spaces $M$ and $N$ are \textit{$CE$-related}, if there exists a topological space $Z$, cell-like maps $q\colon N\rightarrow Z$, $p\colon M\rightarrow Z$, and a homotopy equivalence $f\colon N\rightarrow M$ such that $p\circ f$ is homotopic to $q$.  We refine this definition by declaring that $M$ and $N$ are \textit{CE-related over $X$} if $Z$ in the definition of CE-related is $X$. The next corollary is a refinement of Theorem 6.4 in \cite{DFW}, which asserts that if $n\geq 6$, then,  under the hypotheses of Corollary~\ref{DFWcor}, $M^n$ and $N^n$ are CE-related. The argument in \cite{DFW} works when $X$ is infinite-dimensional, whereas ours requires $X$ to be finite-dimensional. On the other hand, our proof does not use controlled surgery, holds for dimensions $n\geq 4$ rather than $n\geq 6$, and 
is more precise regarding the space over which the approximating manifolds are CE-related, showing that the approximating manifolds are eventually CE-related over the limit space.

In what follows, the notation $\partial{\mathcal{M}}(n,\rho)$ denotes the boundary of $\mathcal{M}(n,\rho)$ in the Gromov--Hausdorff topology.

\begin{cor}
\label{DFWcor}
Fix $n\geq 4$. Assume $X^n$ is a $n$ dimensional compact metric space such that $X\in \partial{\mathcal{M}(n,\rho)}$. Then there exists an $\epsilon>0$ such that, if $M,N\in \mathcal{M}(n,\rho)$ are in an $\epsilon$-neighborhood of $X$ in the Gromov--Hausdorff topology, then $M$ and $N$ are $CE$-related over $X$.
\end{cor}

We organize this paper as follows. In Section~\ref{section preliminaries and historical remarks}, we discuss the preliminaries and make some historical remarks. This section contains essentially no new material. In Section~\ref{section proof of Thm}, we prove Theorem~\ref{main thm} and Corollary~\ref{conjecturebymooreII}. In Section~\ref{section proof of corollaries}, we prove the rest of the corollaries. In Section~\ref{section examples}, we discuss more examples and give more applications which, to the best of our knowledge, are not found in the literature.

\begin{ack}
The authors thank their advisor Fernando Galaz-García, as well as Martin Kerin, Wilhelm Klingenberg,  Mauricio Che, Alexander Jackson, and Julio Quijas Aceves for their support, insights, and suggestions. They are also grateful to Philipp Reiser for a helpful observation and several comments, and to Shmuel Weinberger for valuable comments during the early stages of this project. The authors further thank Wolfgang Lück, Washington Mio, Mark Powell, and Dirk Schütz for sharing their expertise and clarifying several points on topology. M. Alattar is also thankful to Alexander Lytchak for several enlightening mathematical discussions including the Product Structure theorem, cell-like maps, metric geometry and for many helpful comments.

Part of this work was carried out while M.\ Alattar was a research fellow and participant in the Trimester Program on Metric Analysis at the Hausdorff Research Institute for Mathematics in Bonn, where he was supported by the Deutsche Forschungsgemeinschaft (DFG, German Research Foundation) under Germany’s Excellence Strategy – EXC-2047/1 – 390685813. He gratefully acknowledges the opportunity, the excellent working conditions, and the stimulating mathematical discussions.

\end{ack}

\section{Preliminaries}

\label{section preliminaries and historical remarks}

In this section, we recall basic material on the theory of absolute neighborhood retracts (ANR's for short), the theory of homology manifolds, and the theory of cell-like maps. We also recall the basic results on resolutions of homology manifolds and controlled topology that we will use in subsequent sections. For further details, we refer the reader to \cite{Daverman,LacherI,LacherII,homologymanifoldbookhegenbarthfriedrichandcavicchilo}.  

\subsection{Homology Manifolds}

First, we start with the definition of an ANR.

\begin{definition} [ANR]
\label{anr}

A space $M$ is an absolute neighborhood retract (ANR) if, whenever $M$ is topologically embedded as a closed subset of some normal space $Z$, $M$ is a retract of some open subset $U$ of $Z$.

\end{definition}

\begin{exmp}\
    \begin{enumerate}
   
        \item Every manifold is an ANR.
        \item Every Alexandrov space is an ANR \cite{perelman1991alexandrov,siebenmann1972deformation}.
        \item Generalizing the previous example, every finite dimensional Weakly Cone-like  Set (WCS set in the sense of Siebenmann \cite{siebenmann1972deformation}) is an ANR.
        \item  Ricci limit spaces need not be ANR's \cite{menguy}.
        \item Spaces with upper curvature bounds are ANR's \cite{pedro,kramer,LytchakNagano2}.
        \item Locally compact, locally geodesically complete, locally Busemann space of non-positive curvature (BNPC spaces) are ANR's \cite{fujiokaandgu}.

    \end{enumerate}
\end{exmp}

\begin{definition} [Homology Manifold]
\label{homology manifold}

A locally compact space $X$ is a \textit{homology $n$-manifold} if for any $x\in X$,
$$
H_{i}(X,X - x) =
\begin{cases} 
\mathbb{Z}, & \text{if } i = n, \\ 
0, & \text{else}.
\end{cases}
$$
\end{definition}

\begin{definition}[Cell-like map]
\label{Cell-like maps}

Let $X$ and $Y$ be Hausdorff spaces. A map $f\colon X\rightarrow Y$ is a \textit{cell-like map} if it is proper, surjective, and if for every $y\in Y$, and for every neighborhood $V$ of $f^{-1}(y)$, there exists a neighborhood $U$ of $f^{-1}(y)$ in $V$, such that the inclusion $i\colon U\hookrightarrow V$ is null-homotopic.

\end{definition}

Cell-like maps, also referred to as \textit{CE maps}, have found profound applications in many important areas of mathematics. For example, cell-like sets (which are intimately related to cell-like maps) were used by Freedman to prove the $4$-dimensional Poincaré conjecture \cite{4-dpoincare}. In \cite{Lytchakstefancelllikemaps} Lytchak and Wenger used cell-like maps to give an elegant and conceptually simpler proof of the celebrated result due to Bonk and Kleiner \cite{BonkKleiner} on the existence of quasisymmetric parametrizations of linearly locally connected, $2$-regular metric $2$-spheres. Further, Lytchak, Nagano, and Stadler \cite{CAT4manifoldsareEuclidean} used cell-like maps to show that topological $4$-manifolds with globally non-positive curvature are homeomorphic to $\mathbb{R}^4$ \cite{CAT4manifoldsareEuclidean}.

\begin{exmp}
\label{compositionanr's}

The composition of cell-like maps between ANR's is a cell-like map \cite{Daverman}. If the condition ``$\mathrm{ANR}$'' is dropped, then the statement is false \cite{Daverman,Sakai}.

\end{exmp}

\begin{remark}
If $X$ has finite covering dimension and $f\colon X\rightarrow Y$ is cell-like, then it is not necessarily true that $Y$ has finite covering dimension as well.   The space $Y$ can be a Dranishnikov space. That is, $Y$ can have finite cohomological dimension but infinite covering dimension (see also \cite{DydakandWalsh} for related work in dimension two and \cite{Robert-Edwards} for an exposition of the theory and history of cell-like maps).

\end{remark}

Siebenmann showed that cell-like maps arise naturally:

\begin{thm}
[\protect{Siebenmann \cite{approximatingcellularmapsbyhomeossiebenmann}}] \label{siebenmann's theorem}
Let $M^n$ and $N^n$ be topological manifolds of dimension $n\geq 5$. If $f\colon M^n\rightarrow N^n$ is a cell-like map, then $f$ can be approximated uniformly by homeomorphisms.
\end{thm}

Combining Siebenmann's theorem with the fact that the uniform limit of homeomorphisms between ANR's is a cell-like map \cite{LacherI,LacherII}, one obtains that provided $n\geq 5$, the closure of the homeomorphism group $\mathrm{Homeo}(M^n,N^n)$ of two topologically equivalent $n$-manifolds $M^n$ and $N^n$, is equal to the set of cell-like-maps.  For dimension $n=4$, the analogous result was obtained by Quinn \cite{EndsofmapsIII} and for dimension $n=3$, the result was obtained earlier by Armentrout \cite{Armentrout}.

A few years after Siebenmann obtained his theorem, Edwards \cite{Robert-Edwards}, proved a more general result, relaxing the condition of the target space being a manifold to that of satisfying the disjoint disk property, which is a weak notion of transversality asserting that two maps from the $2$-disk to $X$ can be put in general position:

\begin{thm}
[\protect{Edwards \cite{Robert-Edwards}}]\label{Edward's theorem}
Suppose $f\colon M\rightarrow X$ is a cell-like map from an $n$-manifold without boundary onto an ANR $X$. If $n\geq 5$, then the following assertions are equivalent:

\begin{enumerate}
    \item The map $f$ can be approximated by homeomorphisms.
    \item $X$ has the disjoint disk property (DDP).
\end{enumerate}
\end{thm}

\begin{definition}[Resolution]
\label{def-of-resolution}
An $n$-dimensional homology manifold $X^n$ is \textit{resolvable} if there exist an $n$-dimensional manifold $M^n$ and a cell-like map $f\colon M^n\rightarrow X^n$. The map $f$ is a \textit{resolution} of $X$.
\end{definition}

\begin{exmp}
Resolutions are homotopy equivalences. In other words, every resolvable homology manifold is homotopy equivalent to a manifold \cite{Daverman}. 
\end{exmp}

The next example is important for historical reasons, and illustrates both the motivation behind studying cell-like maps and their relationship to homology manifolds.

\begin{exmp}
If $M^n$ is a manifold, the existence of a cell-like map $f\colon M^n\rightarrow X$ implies that $X$ is an $n$-homology manifold  \cite[p.\ 191]{Daverman}. The converse has had long histories. Initially, Cannon, Bryant and Lacher \cite{CannontBryantLacher} established the converse provided that $n\geq 5$ and that the set of non-manifold points have dimension $k$, where $2k+2\leq n$.  In general, the converse is not true: not every ENR homology manifold is resolvable. This result is a consequence of deep work due to Ferry, Bryant, Mio, and Weinberger \cite{topologyofhomologymanifolds}.
\end{exmp}

 Quinn \cite{obstructiontoresolutionofhomologymanifolds} observed that there exists an obstruction to a homology manifold being resolvable:

\begin{thm}[Quinn \protect{\cite{obstructiontoresolutionofhomologymanifolds}}]
\label{quinn'sobstruction}
Let $X^n$ be a connected ANR homology manifold. Then there exists an $I(X)\in 8\mathbb{Z}+1$ such that the following conditions are met:

\begin{enumerate}
    \item    If $U$ is open in $X$, then $I(U)=I(X)$.
    \item $I(X\times Y)=I(X)I(Y)$.
    \item If $\mathrm{dim}(X)\geq 5$ or if $\mathrm{dim}(X)=4$ and $\partial{X}$ is a manifold, then $I(X)=1$ if and only if $X$ is resolvable.
\end{enumerate}

\end{thm}

To conclude this subsection, we mention an essential property of resolvable homology manifolds, due to Quinn: 

\begin{thm}
[\protect{Quinn \cite{QuinnEndsofMaps,EndsofmapsIII}}]
\label{quin equiv}

Let $X$ be a compact metric space of finite dimension at least $4$. Then the following assertions are equivalent:

\begin{enumerate}
    \item $X$ admits a resolution.
    \item $X\times \mathbb{R}^k$ is a manifold for some $k$.
    \item $X\times \mathbb{R}^2$ is a manifold.
\end{enumerate}
\end{thm}

\subsection{Uniqueness of Resolutions}

In \cite{LacherI,LacherII}, Lacher investigated the conjecture asserting the uniqueness of resolutions in the sense that if $X$ admits two resolutions, then there exists a homeomorphism between them. Initially, Bryant and Lacher solved the conjecture in the special case when the set of non-manifold points of the homology manifold is a discrete set \cite{BryantandLacher}.  This situation already encompasses a wide range of spaces. For example, Alexandrov spaces with curvature bounded below that are homology manifolds satisfy such a condition \cite{Wu}. This latter result can also be derived using the theory of stratified spaces \cite{henderson} and work due to Fujioka \cite{fujioka2024alexandrov}. Similarly, every metric space with an upper curvature bound that is a homology manifold is a manifold away from a locally finite subset \cite{LytchakNagano2}. Moreover, every locally Busemann space of non-positive curvature that is a homology manifold has discrete singular set \cite{fujiokaandgu}.

Shortly after Bryant and Lacher's result, Quinn  solved the uniqueness of resolutions conjecture in complete generality:

\begin{thm}
[\protect{Quinn \cite{QuinnEndsofMaps,EndsofmapsIII}}] \label{uniquenessofresolutions}
Assume $X^n$ is a compact metric space of dimension $n\geq 4$ and that there exists two resolutions $f_1\colon M_1^n\rightarrow X$ and $f_2\colon M_2^n\rightarrow X$. Then, for every $\epsilon>0$, there exists a homeomorphism $h\colon M_1^n\rightarrow M_2^n$ such that $d(f_2\circ h,f_1)<\epsilon$.
\end{thm}

\begin{exmp}
If $X$ is infinite dimensional, then uniqueness of resolutions need not hold \cite{DFW}. However, Ferry \cite[Theorem 3]{ferryfiniteness} showed that for any space $X$, there are at most finitely many non-homeomorphic closed $n$-manifolds, $n\neq 3$, that admit CE maps over $X$.  The authors show in \cite{DFW} that there exists an infinite dimensional compactum $X$ and two non-homeomorphic smooth $n$ manifolds $M$ and $N$ that admit CE maps over $X$. This reflects that the infinite dimensional situation is rather different than the finite dimensional setting. For completeness we mention that recently, Daher \cite{Daher} refined the construction above by showing that for any natural numbers $m$ and $n\geq 6$, there are $m$-closed non-homeomorphic simply connected $n$-manifolds which are CE-related.

\end{exmp}

Quinn's result relies heavily on his ``thin $h$-cobordism theorem''. Before stating this theorem, we define the notion of \textit{$(\delta,1)$-connectedness}.

\begin{definition}[($\delta,1)$-connected maps]
Let $X$ be a metric space and let $\delta>0$. A map $e\colon W\rightarrow X$ is \textit{$(\delta,1)$-connected} if, given a relative $2$-complex $(R,S)$, where $i\colon S \hookrightarrow R$ is the inclusion, and a commutative diagram

\[
\begin{tikzcd}
S \arrow[r, "s"] \arrow[d, "i"] & M \arrow[d, "e"] \\
R \arrow[r, "r"']               & X               
\end{tikzcd},
\]
there exists a map $E\colon R\rightarrow M$ such that $E\circ i=s$ and $d(e\circ E,r)<\delta$.

\end{definition}

Roughly speaking, $(\delta,1)$-connectedness ensures that the local fundamental groups are not too wild.

\begin{exmp}
\

\begin{enumerate}
    \item One can easily verify that a homeomorphism is a $(\delta,1)$-connected map for all $\delta>0$. 
    \item More generally, a CE-map between compact ANR's is $(\delta,1)$-connected for all $\delta>0$ \cite{LacherI,LacherII}.

\end{enumerate}

\end{exmp}

 We now state the thin $h$-cobordism theorem (cf. \cite[p. 505]{EndsofmapsIII}).

\begin{thm}
[Thin $h$-cobordism theorem; Quinn \cite{QuinnEndsofMaps,EndsofmapsIII}] Suppose $X$ is a compact metric space with trivial local  fundamental groups. Given $\epsilon>0$, there exists a $\delta>0$ such that if $e\colon W\rightarrow X$ is  $(\delta,1)$-connected and is a proper $(\delta,h)$-cobordism of dimension $n+1\geq 5$, then $W$ has an $\epsilon$-product structure over $X$.

\end{thm}

The term \textit{trivial local fundamental group} is meant in the sense of Freedman and Quinn \cite[p. 89]{freedmanandquinn}.

Quinn's thin $h$-cobordism theorem has found deep applications. For a discussion of some of these, such as a proof of Freedman's $4$-dimensional Poincaré conjecture, we refer the reader to \cite{discembeddingtheorem}. Chapman generalized the thin $h$-cobordism theorem significantly in \cite{chapman}.

\subsection{$\alpha$-Approximation Theorem}

A fundamental result in controlled topology is the \textit{$\alpha$-approximation theorem} due to Chapman and Ferry, which has been useful in  comparison geometry \cite{GPW}. Before stating this theorem, we need a couple of definitions.

\begin{definition}[$\epsilon$-map]
Fix $\epsilon>0$. A map $ f\colon X\rightarrow Y$ between metric spaces is an \textit{$\epsilon$-map} if the diameter of each $f^{-1}(y)$ is at most $\epsilon$.
\end{definition}

\begin{definition} [$\epsilon$-equivalence] \label{controlledequivalence}
Fix $\epsilon>0$. A map $f\colon X\rightarrow Y$ is an $\epsilon$-\textit{equivalence} if there exists a map $g\colon Y\rightarrow X$ such that there exists a homotopy $H_{Y}\colon Y\times I\rightarrow Y$ between $f\circ g$ and $\mathrm{id}_{Y}$ such that $\mathrm{diam}(H_{Y}(y,I))\leq \epsilon$ for all $y\in Y$ and, there exists a homotopy $H_{X}\colon X\times I\rightarrow X$ such that $\mathrm{diam}(fH_{X}(x,I))\leq \epsilon$  for all $x\in X$.

\end{definition}

In other words, an $\epsilon$-equivalence $f$ is a map which has a homotopy inverse and for which the homotopies that witness the equivalence have small tracks in the target space of $f$.

\begin{exmp}
If $X$ and $Y$ are ANR's, then every cell-like map $f\colon X\rightarrow Y$ is an $\epsilon$-homotopy equivalence for all $\epsilon>0$ \cite{LacherI,LacherII}.

\end{exmp}

We will use the following version of Chapman and Ferry's $\alpha$-approximation theorem:

\begin{thm}
[$(\epsilon,\delta)$-approximation theorem; Chapman and Ferry, \cite{ChapmanFerry}] \label{alphaapproximationtheorem}
Let $(M^n,d)$ be a closed metric $n$-manifold with $n\geq 5$. Then, for every $\epsilon>0$, there exists a $\delta>0$ such that if $f\colon M^n\rightarrow N^n$ is a $\delta$-equivalence, then $f$ is $\epsilon$-close to a homeomorphism $h\colon M^n\rightarrow N^n$.
\end{thm}

Chapman and Ferry proved a non-compact, more general version of the statement above. However, we will only concern ourselves with the compact case.

Au \cite{approximatinghomotopyequivalencesbyhomeosdim4}, and Ferry and Weinberger \cite{ferryweinberger}, independently, extended the $\alpha$-approximation theorem to dimension $4$ using different methods. Au overcame the hurdles in the original proof to extend the Chapman and Ferry's arguments to dimension $4$ while Ferry and Weinberger used Quinn's thin $h$-cobordism theorem. We note, however, that both used Quinn's $4$-dimensional thin $h$-cobordism theorem. The three dimensional $\alpha$-approximation theorem is also true, and is due to Jakobsche \cite{Jakobsche} (modulo the Poincaré Conjecture, which has now been resolved by Perelman).

\subsection{Gromov--Hausdorff Convergence and Contractibility Functions}
\label{GH-convergence and contractibility functions}
In this section, we recall basic material  concerning Gromov--Hausdorff convergence and contractibility functions. We also collect fundamental results we will rely on in the proofs of our main theorems. 
   
\begin{definition}[$\epsilon$-commutative diagram]
\label{almost-commutativediagram}

Let $X,Y$ be arbitary spaces and let $Z$ be a metric space. Fix $\epsilon> 0$. Consider a diagram of maps:

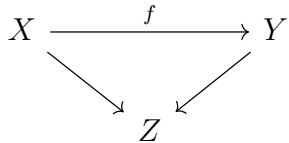
\begingroup
\centering

\begin{tikzcd}
X \arrow[rr, "f"] \arrow[rd] &   & Y \arrow[ld,] \\
                                  & Z &                   
\end{tikzcd}

\endgroup

\noindent The diagram is said to \textit{almost commute with an error} $\epsilon$ if for any $x\in X$, the distance in $Z$, between the images of $x$ under the two maps from the diagram, is at most $\epsilon$. Such a diagram is also called an $\epsilon$-\textit{commutative diagram}.

\end{definition}

Trivially, commutative diagrams are $\epsilon$-commutative for all $\epsilon>0$.

\label{GH-convergenceandcontractibilityfunctions}

\begin{definition}[Almost isometries]
\label{GH-convergence}
Let $(X,d_{X})$ and $(Y,d_{Y})$ be compact metric spaces. Fix $\epsilon>0$. An $\epsilon$-\textit{isometry} $f\colon X\rightarrow Y$ is a (not necessarily continuous) map satisfying the following properties:

\begin{enumerate}
\item $(\epsilon$-surjective): For every $y\in Y$, there exists an $x\in X$ such that $d_{Y}(f(x),y)\leq \epsilon$.
\item $(\epsilon$-distance preserving): For $x_1,x_2\in X$, $|d_{Y}(f(x_1),f(x_2))-d_{X}(x_1,x_2)|\leq \epsilon$.

\end{enumerate}

 The value $\epsilon>0$ is called an \textit{error } of $f$.
\end{definition}

\begin{exmp}
Assume $f\colon X\rightarrow Y$ is an $\epsilon$-isometry. Then, $f$ is an $\epsilon$-map.

\end{exmp}

\begin{definition}[Gromov--Hausdorff Convergence]
\label{GHconv}
Let $\{X_i\}_{i\in \mathbb{N}}$ be a sequence of compact metric spaces and let $X$ be a compact metric space. We say that $X_i$ \textit{Gromov--Hausdorff converges to $X$}, denoted $X_i\rightarrow_{GH}X$, if there exists a sequence of positive real numbers $\{\epsilon_i\}_{i\in \mathbb{N}}$ such that $\lim_{i\rightarrow \infty}\epsilon_i=0$ and a sequence of $\epsilon_i$-isometries $f_i\colon X_i\rightarrow X$.
\end{definition}

We will now define one of the central notions underlying our results.

\begin{definition}[Contractibility function]
\label{contractibility function}

A function $\rho\colon [0,R]\rightarrow [0,\infty)$ is a \textit{contractibility function} if $\rho(t)\geq t$ for all $t$, $\rho(0)=0$ and $\rho$ is continuous at $0$. The function $\rho$ is a contractibility function for a metric space $X$ if for every $t\leq R$ and $x\in X$, the ball $B_t(x)$ contracts in the ball $B_{\rho(t)}(x)$. That is, the inclusion $i\colon B_t(x)\hookrightarrow B_{\rho(t)}(x)$ is null-homotopic.

\end{definition}

We will also rely on the result:

\begin{thm}[Petersen \cite{petersenfinitenesstheoremformetricspaces}] \label{petersen} Let $\mathcal{C}$ be a family of compact $n$-dimensional metric spaces admitting a uniform local contractibility function $\rho$. For any $\epsilon>0$, there exists a $\delta>0$ such that if $X,Y\in \mathcal{C}$ and $d_{GH}(X,Y)\leq \delta$ then $X$ and $Y$ are $\epsilon$-homotopy equivalent.
\end{thm}

We now define a central class of spaces we will deal with. 

\begin{definition}[Locally geometrically contractible]
\label{LGC}
Given a metric space $X$, and a contractibility function $\rho$ on $X$, we say that $X$ is $\mathrm{LGC}(n,\rho)$ if the following condition holds:
\smallskip

\begin{itemize}
    \item[(LGC)] For each map $\alpha\colon \mathbb{S}^k \rightarrow X$ with $k\leq n$ and $\mathrm{diam}(\alpha(\mathbb{S}^k))<R$, there exists a map 
    \[
    \overline{\alpha} \colon D^{k+1} \rightarrow X
    \]
    extending $\alpha$ and such that the image has diameter smaller than $\rho(\mathrm{diam}(\alpha(\mathbb{S}^k))$.
\end{itemize}
\smallskip

We say that $X$ is $\mathrm{LGC}(\rho)$ if $X$ is in $\mathrm{LGC}(n,\rho)$ for some $n$ and $\rho$.
\end{definition}

\begin{exmp}[\cite{Moore}]
If a space $X$ is in $\mathrm{LGC}(n,\rho)$ for some $n$ and $\rho$, then it need not be the case that $X$ has finite covering dimension \cite{Moore}. The space $X$ can be a Dranishnikov space \cite{DranishnikovI,DranishnikovII}.

\end{exmp}


The following result shows that the class $\mathrm{LGC}(\rho)$ is large (see \cite{Borsuk,Moore}).

\begin{prop}
\label{anrcharac}

Let $X$ be a compact $n$-dimensional metric space, then $X$ is an ANR if and only if  $X\in \mathrm{LGC}(n,\rho)$ for some contractibility function $\rho$.

\end{prop}

In \cite{Moore}, Moore showed  the following result (cf.\ \cite{ferryandborus}).

\begin{thm}[Moore \cite{Moore}]
\label{moore'sthm}

Let $M^n$ be a closed $n$-manifold and $f\colon M\rightarrow X$ a CE map, where $X$ is compact. Then there exists a metric $d$ on $X$, a contractibility function $\rho$ and a continuous path 
\[\omega\colon [0,1]\rightarrow \mathrm{LGC}(n,\rho)
\]
such that $\omega(t)$ is homeomorphic to $M$ for $t\in [0,1)$ and $\omega(1)=(X,d)$.

\end{thm}

Moore's proof \cite[p. 414]{Moore} relies on the fact that under the hypothesis above, the mapping cylinder $M_{f}=\frac{X\times I \coprod Y}{(x,1)\sim f(x)}$ of $f$ is metrizable and hence therefore so is $X$.

In particular, Moore does not assume that $X$ is a  metric space and, instead, constructs a metric simultaneously on $X$ and $M$. For the purposes of this paper, we need a slight refinement of the above theorem. Namely, we require that if $X$ is metric, then from this, one can construct metrics on $M$.  In view of Moore's proof \cite{Moore}, it suffices to establish the following natural Lemma (which holds in greater generality). 

Before proceeding, the authors would like to thank Alexander Lytchak for suggesting the simplification of proof to an earlier proof.

\begin{lemma}\label{thm}
Let $(X,d_{X})$ and $(Y,d_{Y})$ be compact metric spaces and $f\colon X\rightarrow Y$ a cell-like map. Then there exists a metric $\tilde{d}$ on the mapping cylinder $M_{f}$  such that $(Y,d_{Y})$ isometrically embeds into $(M_{f},\tilde{d})$ via the natural embedding. 

\end{lemma}

Before proceeding with the proof, we mention the following result due to Hausdorff (cf. \cite{shortproofofhausdorff} and \cite{Sakai})

\begin{thm}[Hausdorff Metric Extension Theorem]\label{hausdorffmetrization}
Let $X$ be a metrizable topological space and $A\subseteq X$ closed. Every metric on $A$ inducing the subspace topology on $A$, can be extended to a metric on $X$.

\end{thm}

\begin{proof} (Of Lemma \ref{thm})
 Urysohn's theorem gives that $M_{f}$ is metrizable. To find the metric $\tilde{d}$ on the mapping cylinder, one merely needs to apply Hausdorff's extension theorem to $M_f$ and $Y$. For the convenience of the reader, we supply the details. There exists a natural embedding $i\colon Y\hookrightarrow M_f$. Thus one can endow $i(Y)$ with a metric for which $i(Y)$ then becomes isometric to $Y$. Thus one applies Hausdorff's extension theorem, Theorem \ref{hausdorffmetrization} to extend this metric to a metric on $M_f$.

\end{proof}

Combining Moore's proof \cite{Moore}, with the previous lemma, one immediately obtains the following Corollary.

\begin{corollary}\label{refinementofmoore'sthm}
Let $M^n$ be a closed $n$-manifold and $f\colon M\rightarrow (X,d)$ a CE map.   Here $(X,d)$ is a compact metric space. Then there exists a contractibility function $\rho$ and a continuous path $\omega \colon [0,1]\rightarrow LGC(n,\rho)$ such that $\omega(t)$ is homeomorphic to $M$ for $t\in [0,1)$ and $\omega(1)=(X,d)$.
\end{corollary}

\begin{proof}
Let $M_f$ be the mapping cylinder of $f$. By Lemma \ref{thm}, there exists a metric $\tilde{d}$ on $M_f$ for which $(X,d)$ isometrically embeds into $(M_f,\tilde{d})$. Following \cite{Moore}, consider the projection map $\mathrm{proj}\colon M_f\rightarrow [0,1]$. The map $\omega(t)= \mathrm{proj}^{-1}(t)$ yields the desired path.

\end{proof}

We will rely frequently on the following refinement of Moore's theorem:

\begin{thm}[Ferry and Okun \cite{ferryandborus}]
\label{ferryandokun}

Let $ f\colon M\rightarrow (X,d)$ be a CE map, where $(X,d)$ is a geodesic metric space, and $M$ is a connected, closed manifold of dimension at least $3$ that is smoothable. Then there exists a sequence of Riemannian metrics $g_i$ on $M_i$ such that $(M,d_{g_i})\rightarrow_{GH}(X,d)$. Moreover, the sequence $\{(M,d_{g_i})\}_{i=1}^{\infty}$ admits a uniform contractibility function $\rho$.

\end{thm}

We now give examples of which spaces one might obtain under convergence with a uniform contractibility function.

\begin{exmp}[Ferry and Okun, \cite{ferryandborus}]
There exists an infinite dimensional space $X$, a sequence of Riemannian metrics $\{g_i\}_{i=1}^\infty$ on $\mathbb{S}^5$, and a uniform contractibility function $\rho$ for $\{(\mathbb{S}^5,d_{g_i})\}_{i=1}^\infty$ such that $(\mathbb{S}^5,d_{g_i})\rightarrow_{GH} X$.  This example illustrates that the limit of finite-dimensional spaces might be infinite-dimensional even with a uniform contractibility function.
\end{exmp}

\begin{exmp}[Moore, \cite{Moore}]
There exists a uniform contractibility function $\rho$ for a family of compact metric manifolds $X_i$ homeomorphic to $\mathbb{S}^3$ and a finite dimensional non-manifold $X$ such that $X_i\rightarrow_{GH}X$.
\end{exmp}

\begin{exmp}[Ghastly examples]
\label{GhastlyExamples}
Combining the results of Davermann and Walsh \cite{ghastly} and Moore \cite{Moore}, one can find, for $n\geq 3$, spaces $X^n$ and metrics $d_i$ on $\mathbb{S}^n$ such that $(\mathbb{S}^n,d_i)\rightarrow_{GH}X^n$ and for which there does not exists a cell-like map from $X^n$ onto any $n$-manifold. In particular $X^n$ cannot be an $n$ manifold. The space $X^n$ satisfies the following properties:

\begin{enumerate}
    \item $X^n$ does not contain any ANR of dimensions strictly between $1$ and $n$.
    \item For each map $f\colon D^2\rightarrow X$ where $f|_{\mathbb{S}^1}$ is an embedding,  $f(D^2)$ has non-empty interior in $X$.
\end{enumerate}

\end{exmp}

The examples above suggest that under  non-collapsed convergence, in which the approximating manifolds admit a uniform contractibility function, the limit is a homology manifold. Indeed, this is always true:

\begin{thm}[Grove, Petersen, and Wu \cite{GPW}] \label{GPWthm}
Let $\{X_i\}_{i=1}^{\infty}$ be a sequence of closed metric $n$-manifolds admitting a uniform contractibility function. If $X_i\rightarrow_{GH}X$ and $X$ is finite dimensional, then $X$ is a resolvable homology manifold.

\end{thm}

\subsection{Product Structure Theorem}

In this section, we will first recall the \textit{product structure theorem}, a fundamental result that we will rely on. We refer the reader to \cite{QuinnProductStructureTheorem} and \cite{Kirby-Siebenmann} for proofs of this theorem. Then, we will define the Kirby--Siebenmann invariant and mention some of its properties. In this section, we will only concern ourselves with the categories $\mathrm{DIFF}$ and $\mathrm{PL}$. 

\begin{thm}[Product structure theorem] 
\label{Product Structure Theorem}
Suppose $M^n$ is a metrizable topological manifold without boundary and of dimension $n\geq 5$. Let $\Gamma$ be a $\mathrm{DIFF}$ or $\mathrm{PL}$ structure on $M\times \mathbb{R}^q$ for some $q\geq 1$. Then, there exists a concordant $\mathrm{CAT}$ structure $\Sigma \times \mathbb{R}^q$ on $M\times \mathbb{R}^q$, where $\Sigma$ is a $\mathrm{DIFF}$ or $\mathrm{PL}$ structure on $M$.

\end{thm}

We will not define what a ``concordance'' is, as it is not necessary for this paper. Instead, we refer the reader to \cite[Essay IV]{Kirby-Siebenmann}. 

The following theorem addresses the case $n=4$.

\begin{thm}[Kirby--Siebenmann invariant \cite{Kirby-Siebenmann}]
\label{kirby-siebenmann invariant}
Suppose  $M^n$ is a topological manifold without boundary and $n\geq 5$. Then there exists an obstruction, $\mathrm{ks}(M)\in H^4(M,\mathbb{Z}_2)$ such that $\mathrm{ks}(M)=0$ if and only if $M$ admits a $\mathrm{PL}$ structure.  If $n=4$, then $\mathrm{ks}(M)=0$ implies that $M\times \mathbb{R}$ is smoothable.
\end{thm}

The obstruction $\mathrm{ks}(M)$ in the previous theorem is known as the \textit{Kirby--Siebenmann invariant}.
For some of the properties of the Kirby--Siebenmann invariant, we refer the reader to \cite{applicationsofKSindim4} and references therein.

\begin{remark}
The Kirby--Siebenmann invariant will play a very minor role in this paper. Its usefulness will be apparent in section \ref{section examples}, where it will allow us to reformulate results in a more appealing manner (see Corollary \ref{obstructionvanishing}).
\end{remark}

\begin{exmp}
\label{E_8mfld}
The product structure theorem in dimension $n=4$ is false in general. 
Indeed, there exists a closed topological $4$-manifold $M^4$ such that $M$ is not smoothable, but $\mathrm{ks}(M)=0$ and hence, $M\times \mathbb{R}$ is smoothable. The manifold $M^4$ is constructed by the so-called $E_{8}$ manifold, due to Freedman \cite{4-dpoincare}. More precisely, $E_{8}\# E_{8}$ is not smoothable \cite{DonaldsonI,DonaldsonII}, but its Kirby--Siebenmann invariant vanishes. We will use this example in section \ref{section examples}.
\end{exmp}

\section{Proofs of Theorem \ref{main thm}}
\label{section proof of Thm}

\begin{proof}[Proof of Theorem \ref{main thm}]
\label{firstproof}
Let $\{X_i^n\}_{i=1}^{\infty}$ be a sequence of closed metric topological $n$-manifolds with a uniform contractibility function $\rho$ converging to $X$ in the Gromov--Hausdorff sense. By \cite{GPW}, $X$ is a resolvable homology manifold. In particular,  $X$ is resolvable by some manifold, say $M$. We will show that for large enough $i$, $X_i$ is homeomorphic to $M$. This follows from combining the results in \cite{GPW} with Corollary \ref{refinementofmoore'sthm} as follows (the following argument will also be used in Section \ref{section examples}). By Corollary ~\ref{refinementofmoore'sthm}, there exists a sequence $\{d_i\}_{i=1}^\infty$ of metrics  on $M$ such that $(M,d_i)\rightarrow_{GH}X$ and each $(M,d_i)$ is in $\mathrm{LGC}(n,\sigma)$ for some fixed contractibility function $\sigma$. Denote $(M,d_i)$ by $M_i$.

By intertwining the sequences $\{M_i\}_{i=1}^{\infty}$ and $\{X_i^n\}_{i=1}^{\infty}$, we will construct a sequence $\{Z_i\}_{i=1}^{\infty}$ of $n$-manifolds that are in some $\mathrm{LGC}(n,\rho')$. For example, define $\{Z_i\}_{i=1}^{\infty}$ by setting $Z_{\frac{i^2}{2}+\frac{3i}{2}}=M_{\frac{i^2}{2}+\frac{3i}{2}}$  for $i\geq 1$  and $Z_k=X_k$ otherwise. The maximum of the contractibility functions $\rho$ and $\sigma$, on the appropriate domain, yields a uniform contractibility function $\rho'$ for the sequence $\{Z_i\}_{i=1}^\infty$. Observe that $Z_i\rightarrow_{GH}X$ and that the $Z_i$ are all in $\mathrm{LGC}(n,\rho')$.  The Grove--Petersen--Wu stability theorem \cite{GPW,Grove-Petersen-Wu-erratum} implies that, for large $i,j$, the space $Z_i$ is homeomorphic to $Z_j$. In particular, $X_i$ is homeomorphic to $M$. Since $X$, $X_i$, and $M$ are all ANR's, it follows that $X$ is a cell-like image of $X_i$ (see Example \ref{compositionanr's}). Finally, the fact that $f$ can be taken to be open follows from the results due to Walsh \cite{Walshopenmappings} (cf. \cite[p. 489]{Sakai}).
\end{proof}

As mentioned in the introduction, we now offer a more classical approach proof to Theorem \ref{main thm}, relying on a more explicit procedure for constructing resolutions. In particular, the next proof gives a more concrete method of constructing the desired resolution.  Further, as mentioned in the introduction, the ideas appear in several different contexts.

Throughout the next argument, we will denote by $\epsilon_i$ arbitrary positive constants that tend to $0$ as $i\rightarrow \infty$. In particular, any positive value proportional to $\epsilon_i$ will also be denoted by $\epsilon_i$.  Sufficiently small positive values will be denoted by $\epsilon$.

\begin{proof}
(Second proof of Theorem \ref{main thm})
    
    \label{secondproof}
We stress that we will follow the proof of the main convergence result in \cite{GPW} and then combine it with an argument in \cite[p.\ 444]{topologyofhomologymanifolds}.  For large $i$, let $f_i\colon X_i\rightarrow X$ be an $\epsilon_i$-homotopy equivalence that is almost distance-preserving up to an error $\epsilon_i$. Such maps exist by \cite{petersenfinitenesstheoremformetricspaces}. Let $\epsilon>0$. By the $\alpha$-approximation theorem, there exists an $i(\epsilon)$ such that for all $i\geq i(\epsilon)$, each $(\mathrm{f}_i,\mathrm{id},\mathrm{id})\colon X_i\times \mathbb{S}^1\times \mathbb{S}^1 \rightarrow X\times \mathbb{S}^1\times \mathbb{S}^1$ is $\epsilon$-close to a homeomorphism $F_i\colon X_i\times \mathbb{S}^1\times \mathbb{S}^1\rightarrow X\times \mathbb{S}^1\times \mathbb{S}^1$. Thus, we obtain, for large $i,j$, a homeomorphism $H\colon X_i\times \mathbb{S}^1\times \mathbb{S}^1 \rightarrow X_j\times \mathbb{S}^1\times \mathbb{S}^1$ such that $F_j\circ H=F_i$. Writing $\mathbb{T}^2$ as $\mathbb{S}^1\times \mathbb{S}^1$, for large $i,j$ consider the following diagram:

\begin{center}
\begin{tikzcd}
X_i\times \mathbb{S}^1\times \mathbb{R} \arrow[rr, "\hat{H}"] \arrow[d] &                                & X_j\times \mathbb{S}^1\times \mathbb{R} \arrow[d] \\
X_i\times \mathbb{T}^2 \arrow[rr, "H"] \arrow[rd, "F_i"]                &                                & X_j\times \mathbb{T}^2 \arrow[ld, "F_j"']         \\
                                                                        & (X\times \mathbb{S}^1)\times \mathbb{S}^1 \arrow[d] &                                                   \\
                                                                        & X\times \mathbb{S}^1           &                                                  
\end{tikzcd}
.
\end{center}

The vertical maps are projections and the map $\hat{H}$ is a lift. As in \cite{GPW,QuinnEndsofMaps}, choose $t_0\in \mathbb{R}$ such that $\hat{H}_i(X_i\times \mathbb{S}^1\times 0)$ does not meet $X_j\times \mathbb{S}^1\times t_0$. Let $W$ denote the region between $\hat{H}_i(X_i \times \mathbb{S}^1\times 0)$ and $X_j\times \mathbb{S}^1\times t_0$. Indeed, following \cite{QuinnEndsofMaps,GPW}, $W$ is a $(\delta(\epsilon),h)$-cobordism and is $(\delta(\epsilon),1)$-connected over $X\times \mathbb{S}^1$, where $\delta(\epsilon)\rightarrow 0$ as $\epsilon\rightarrow 0$ (cf.\ also \cite[p.\ 394]{Wustructureofalmostnonnegative}).

Thus, provided $\epsilon$ is small enough, for any $\gamma>0$,  there exists a homeomorphism $h_i\colon \hat{H}_i(X_i\times \mathbb{S}^1\times 0)\times I\rightarrow W$ that is a $\gamma$-product over  $X\times \mathbb{S}^1$ and such that $h_i|_{\hat{H}_i(X\times \mathbb{S}^1\times 0)\times 1}$ is the identity. As homeomorphisms preserve boundaries, we obtain the following $\gamma$-commutative diagram (see Definition \ref{almost-commutativediagram}):

\begin{center}
\begin{tikzcd}
X_i\times \mathbb{S}^1 \arrow[rr, "\hat{h}_0"] \arrow[rd] &                      & X_j\times \mathbb{S}^1 \arrow[ld] \\
                                                                                           & X\times \mathbb{S}^1 &                                                                  
\end{tikzcd}
.
\end{center}

Repeating the argument in the preceding paragraphs (see for example \cite[p. 395-396]{Wustructureofalmostnonnegative} or \cite[p. 211]{GPW,Grove-Petersen-Wu-erratum}), we can remove the last circle factor in the products in the preceding almost commutative diagram and obtain a homeomorphism $\hat{h}_1\colon X_i\rightarrow X_j$ such that the following diagram almost commutes with an ``error'' that is small (see Definition \ref{almost-commutativediagram}) for all large $i,j$:

\begin{center}
\begin{tikzcd}
X_i \arrow[rr, "\hat{h}_1"] \arrow[rd, "f_i"'] &   & X_j \arrow[ld, "f_j"] \\
                                                        & X &                               
\end{tikzcd}

\end{center}

Therefore, for every $\epsilon>0$, there exists an $i(\epsilon)$ such that for all $i,j\geq i(\epsilon)$, we have homeomorphisms $h_i\colon X_i\rightarrow X_j$ such that $d(f_j\circ h_i,f_i)<\epsilon$.  Following \cite[p. 444]{topologyofhomologymanifolds}, one sees that for $\epsilon_i\rightarrow 0$ such that $\sum_i\epsilon_i<\infty$, there exists a sequence $\{\eta_i\}_{i=1}^\infty$ and homeomorphisms $h_i\colon X_{\eta_i}\rightarrow X_{\eta_{i+1}}$ such that $d(f_{\eta_{i+1}}\circ h_i, f_{\eta_i})<\epsilon_i$. Thus, the sequence  $f_{\eta_i}\circ h_{i-1}\circ ....\circ h_0$ converges to a map $f$.  Since each $f_{\eta_i}\circ h_{i-1}\circ....\circ h_0$ is a $\delta_i$-equivalence, where $\delta_i\rightarrow 0$, it follows that $f$ is a CE-map. Hence, $f$ is a resolution. 

\end{proof}

\section{Proof of Corollaries~\ref{conjecturebymooreII} and \ref{obstructionvanishing}--\ref{DFWcor}}
\label{section proof of corollaries}

\begin{proof}[Proof of Corollary~\ref{conjecturebymooreII}]  Let $X_i= \omega(1-1/i)$ for $i=1,2,\ldots$ Since $\omega$ is continuous, $X_i\rightarrow_{GH}X$, where $X=\omega(1)$. Then, by \cite[Proposition 3.1]{Moore}, $X_i$ is homeomorphic to $\omega(0)$. By Theorem \ref{main thm}, $X$ is the cell-like image of some $X_i$. Since $X$ and $X_i$ are ANR's, it follows that $X$ is the cell-like image of $\omega(0)$.
\end{proof}

\begin{proof}[Proof of Corollary \ref{obstructionvanishing}]
Since $X^n$ is resolvable, there exists a manifold $M^n$ and a cell-like map $f\colon M^n\rightarrow X^n$. Note that $f$ is a controlled homotopy equivalence, i.e., an $\epsilon$-homotopy equivalence for all $\epsilon>0$ (see \cite{LacherI,LacherII}). Define $\Delta(X^n)$ to be the Kirby--Siebenmann invariant of $M^n$. 
Quinn's uniqueness of resolutions (Theorem \ref{uniquenessofresolutions}) asserts that any two resolutions of $X$ are homeomorphic. Therefore, $\Delta(X^n)$ does not depend on the choice of resolution and is well-defined. Hence, $\Delta(X)=0$ if and only if $X$ is resolvable by a PL-manifold. Hence the result now follows from Corollary \ref{approxbyPL}.
\end{proof}

\begin{proof}[Proof of Corollary \ref{corollary smoothable}]
We will show $(1)\implies (2)\implies (3)\implies(1)$. 

Suppose $X$ is resolvable by a smooth manifold $M$. Then, by  Theorem \ref{siebenmann's theorem}, $X\times \mathbb{R}^k$ is smoothable for $k\geq 2$.

Assume now that $X\times \mathbb{R}^k$ is smoothable for some $k\geq 2$. Then Theorem \ref{quin equiv} implies that $X$ is resolvable by a manifold $M$. Hence, Theorem \ref{siebenmann's theorem} once again yields that $X\times \mathbb{R}^k$ is homeomorphic to $M\times \mathbb{R}^k$. Since $X\times \mathbb{R}^k$ is smoothable, the Product Structure Theorem  (Theorem \ref{Product Structure Theorem}) implies that $M$ is smoothable. Thus (3) follows from applying Theorem \ref{ferryandokun}.  

Finally, assume $(3)$ holds. By Theorem \ref{main thm}, $X$ is resolvable by some $X_i$, and $X_i$ is by assumption smooth.
\end{proof}

We mention the following remark.
\begin{remark}
\label{low dimensions}

In dimensions $5\leq n \leq 7$, every PL manifold admits a smooth structure \cite{Milnor}. Thus, if $(X^n,d)$ is a compact geodesic metric space that is an ANR resolvable homology manifold, then in these dimensions, $\Delta(X^n) = 0$ if and only if $(X^n,d)$ is a Gromov--Hausdorff limit of Riemannian  $n$-manifolds with a uniform contractibility function.
\end{remark}

\begin{proof}[Proof of Corollary \ref{DFWcor}]  If $\epsilon$ is small enough, then by Theorem \ref{main thm},  there exist resolutions $c_1\colon M\rightarrow X$ and $c_2\colon N\rightarrow X$. Quinn's uniqueness of resolutions theorem (Theorem \ref{uniquenessofresolutions}) now implies that for any $\delta>0$,  there exists a homeomorphism $f\colon N\rightarrow M$ such that $c_1\circ f$ and $c_2$ are $\delta$-close. Since $X$ is an ANR, if $\delta$ is small enough, $c_1\circ f$ and $c_2$ are homotopic.
\end{proof}

\section{Examples and Further Applications}
\label{section examples}

\begin{exmp}
\label{counterexampleindimension4}
This example demonstrates the analog of Corollary \ref{corollary smoothable} is false in dimensions $4$. Indeed, observe that since the Kirby--Siebenmann invariant of $X=E_{8}\# E_{8}$ vanishes (see Example \ref{E_8mfld}), it follows that $X\times \mathbb{R}$ is smoothable. Consequently, by \cite[Theorem 7]{Bing}, or more generally \cite{bingII}, $X$ admits a geodesic metric. However, if we have closed connected Riemannian manifolds $\{X_i^4\}_{i=1}^\infty$ in some $\mathcal{M}(n,\rho)$ such that $X_i^4\rightarrow_{GH}X^4$, then by the $\alpha$-approximation theorem in dimension $4$ (see discussion following Theorem \ref{alphaapproximationtheorem}), $X^4$ is smoothable. This is a contradiction.

\end{exmp}

Recall that $\mathcal{M}(n,\rho)$ denotes the class of closed (metric) $n$-manifolds with a uniform contractibility function $\rho$ that are in $\mathrm{LGC}(\rho)$. We denote by $\mathcal{C}(n,\rho)$ an arbitrary precompact (in the Gromov--Hausdorff topology) subclass of $\mathcal{M}(n,\rho)$ consisting of all elements $X\in \mathcal{M}(n,\rho)$, such that the closure is compact, and the limit points are non-branching geodesic metric spaces admitting a qualitatively non-degenerate (Borel) measure with good transport behavior (see \cite{Kell-Studied,Garcia-Kell-Mondino,Cavalletti-Husemann,SR22} for definitions of these statements and more).

The next example combines several ideas in this paper by characterizing the property of being a finite dimensional homology manifold with the existence and convergence of a certain class of infinite dimensional spaces.

\begin{exmp}
Let $X^n$ be a compact Alexandrov space. Then $X$ is a homology manifold if and only if there exists $\{X_i\}_{i=1}^{\infty}$ in some $\mathcal{C}(n,\rho)$ such that  $\mathbb{P}_2(X_i)\rightarrow_{GH} \mathbb{P}_2(X)$.   The space $\mathbb{P}_2$ denotes the $2$-Wasserstein space (see \cite{LottandVillani}). The point here is that $\mathbb{P}_2(X)$ and $\mathbb{P}_2(X_i)$ have infinite Hausdorff dimension, whereas $X$ is finite dimensional. Indeed, we may suppose that $n\geq 4$.  Since $X$ is a manifold on an open and dense subset, $I(X)=1$. Here, $I(X)$ is Quinn's obstruction (see Theorem \ref{quinn'sobstruction}). Thus, the forward direction follows from combining Moore's theorem (Theorem \ref{moore'sthm}) with Corollary 4.3 in \cite{LottandVillani}. The converse follows from the results in \cite{me}. This example is true in all dimensions.

\end{exmp}

It is an open question as to whether every topological manifold admits an Alexandrov metric \cite[p.  4]{fernandoandmassoumeh}. However, the next example shows that there are plenty of resolvable finite dimensional homology manifolds that do not admit an Alexandrov metric. Namely, the next example shows that there are examples of homology manifolds with precisely one non-manifold point that cannot admit an Alexandrov metric.

\begin{exmp}\label{non-alexandrovmetric}
We will show that there exists a $3$-dimensional homology manifold that is resolvable, with only one non-manifold point that does not admit an Alexandrov metric. Indeed, let $\alpha$ be the Fox-Artin arc in $\mathbb{S}^3$. Then, $X=\mathbb{S}^3/\alpha$ is a homology $3$-dimensional manifold that is not a manifold \cite{Moore,Mooresthesis}. We give a short proof that $X$ cannot admit an Alexandrov metric using the theory of stratified spaces. If $X$ were to admit such a metric, it follows from the main result in \cite{fujioka2024alexandrov}, that $X$ is a CS set in the sense of Siebenmann \cite{siebenmann1972deformation}. However, it follows from \cite[Theorem 2]{henderson} that $X$ will thus consist of only manifold points. However, this cannot happen, since $X$ has a non-manifold point (only one). One can also use the results in \cite{Wu} as a substitute to \cite{henderson}.

\end{exmp}

\begin{remark}
    The argument above can also be used to give another more general proof of a result due to Wu \cite{Wu}. Namely that if $X$ is an Alexandrov space that is a homology manifold then $\mathrm{dim}S(X)\leq 0$, where $\mathrm{S}(X)$ denotes the set of singular points (i.e. non-manifold points).

    \end{remark}

The authors thank Phillip Reiser for bringing to their attention the content of the following example, which further reinforces the fact that being approximated by manifolds with a uniform contractibility function is a flexible notion.

\begin{exmp}
\label{vitali}
This example emphasizes two things, first that the phenomenon of convergence with a uniform contractibility function is rather flexible. Second, this example demonstrates that the approximating manifolds in Corollary \ref{corollary smoothable} cannot be taken to have a uniform lower (sectional) curvature bound, even if $X$ is a sufficiently nice topological space. More precisely, there exists an Alexandrov space $X^n$ which cannot be obtained as a Gromov--Hausdorff limit of Riemannian manifolds $X_i^n$ such that $\mathrm{sec(X_i)}\geq k'$, but which can be obtained as the Gromov--Hausdorff limit of Riemannian $n$-manifolds with a uniform contractibility function. Moreover, the Alexandrov metric on $X$ can be uniformly approximated by Riemannian metrics. Let $\Sigma^3$ be the Poincaré homology sphere. Consider the double suspension $S^2 \Sigma^3$. Then, $S^2\Sigma^3$ is homeomorphic to $\mathbb{S}^5$ by Edward's double suspension theorem. According to \cite{kapovitch-regularity}, there exists an Alexandrov metric on $ S^2\Sigma^3$ which makes it non-smoothable (in the sense of \cite{kapovitch-regularity}). However, since $\mathrm{ks}(S^2\Sigma^3)=0$, the non-smoothable metric on $S^2\Sigma^3$ can be approximated uniformly by Riemannian metrics with a uniform contractibility function \cite{ferryandborus}.
\end{exmp}

\begin{exmp}
\label{diffeostabilitycounterexample}
The diffeomorphism stability conjecture \cite{grovewilhelm,pro2020stability} is false if one replaces the lower curvature bound condition with a uniform contractibility function condition. This is a natural generalization of the diffeomorphism stability conjecture in view of the fact that the class consisting of closed Riemannian $n$-manifolds with sectional curvature bounded below by $k$, diameter bounded above by $D>0$, and volume bounded below by $v>0$, admits a uniform contractibility function (see \cite{GPW}).

The technique used to demonstrate this is used in the first proof of Theorem \ref{main thm}. More precisely, let $\Sigma^7$ be an exotic $7$-sphere and let $\mathbb{S}^7$ be the unit round sphere. By Theorem \ref{ferryandokun}, there exists a sequence of Riemannian metrics $g_i$ on $\Sigma^7$ such that $(\Sigma^7,g_i)\rightarrow_{GH} \mathbb{S}^7$ and the sequence $\{(\Sigma^7,g_i)\}_{i=1}^\infty$ has a uniform contractibility function $\rho_1$. On the other hand, $\mathbb{S}^7$ has a uniform contractibility function $\rho_2$. Thus intertwining the sequences as in the first proof of Theorem \ref{main thm}, one gets a sequence $\{Z_{i}^n\}_{i=1}^{\infty}$ consisting of smooth $n$-manifolds with a uniform contractibility function, such that $Z_i^n\rightarrow_{GH}\mathbb{S}^7$,  and such that for all large $i,j$, $Z_i$ and $Z_j$ homeomorphic, but it is not true that for all large $i,j$, $Z_i$ and $Z_j$ are diffeomorphic.

Notice that although the exotic sphere is not diffeomorphic to the standard sphere, it is smoothly homeomorphic to it. In other words, there exists a homeomorphism between $\Sigma^7$ and $\mathbb{S}^7$ that is smooth in one direction but not the other (see for example the introduction in \cite{SiebenmannScharlemannII}).

In light of this, it is thus natural to inquire whether, in the generalized diffeomorphism stability above, one could replace diffeomorphisms between the terms of the sequence,  by smooth homeomorphisms between the approximating manifolds in the tail of the sequence. This is also not true. In \cite{SiebenmannScharlemannII}, the authors use a closed $6$-dimensional manifold ${T}(\beta)$, due to Siebenmann \cite{topmnfldssiebenmann}, that is homeomorphic to the $6$ dimensional Torus, yet is not diffeomorphic to it. They showed that a smooth homeomorphism  $T(\beta)\rightarrow \mathbb{T}^6$ would then result in a diffeomorphism between the two manifolds. Hence using these two manifolds, the results of Ferry and Okun \cite{ferryandborus} and the ``alternating trick'' above,  one shows that this inquiry is false.

\end{exmp}

\printbibliography

\end{document}